\documentclass[12pt]{amsart}

\usepackage{amssymb}
\usepackage[cmtip,all]{xy}
\usepackage{verbatim}
\usepackage{upref}
\usepackage[usenames,dvipsnames]{color}
\usepackage{url}
\usepackage[normalem]{ulem}

\newtheorem{thm}{Theorem}[section]
\newtheorem*{thm*}{Theorem}
\newtheorem{lem}[thm]{Lemma}

\newtheorem{prop}[thm]{Proposition}

\theoremstyle{definition} 
\newtheorem{defn}[thm]{Definition}
\newtheorem{ex}[thm]{Example}

\theoremstyle{remark} 
\newtheorem{rem}[thm]{Remark}

\numberwithin{equation}{section}

\newcommand{\thmref}[1]{Theorem~\textup{\ref{#1}}}

\newcommand{\lemref}[1]{Lemma~\textup{\ref{#1}}}
\newcommand{\propref}[1]{Proposition~\textup{\ref{#1}}}

\newcommand{\exref}[1]{Example~\textup{\ref{#1}}}

\newcommand{\KK}{\mathcal K}

\newcommand{\RR}{\mathcal R}

\newcommand{\F}{\mathbb F}

\renewcommand{\bar}{\overline}
\newcommand{\what}{\widehat}
\newcommand{\wilde}{\widetilde}

\newcommand{\<}{\langle}
\renewcommand{\>}{\rangle}

\newcommand{\ann}{^\perp}
\newcommand{\pann}{{}\ann}

\newcommand{\id}{\text{\textup{id}}}

\DeclareMathOperator*{\spn}{span}
\DeclareMathOperator*{\clspn}{\overline{\spn}}
\DeclareMathOperator{\ind}{Ind}
\DeclareMathOperator{\dashind}{\!-Ind}

\newcommand{\righttext}[1]{\quad\text{#1 }}

\newcommand{\cs}{\mathbf{C}^*}

\newcommand{\co}{\mathbf{Coact}}

\newcommand{\cp}{\textup{CP}}

\begin{document}
\title{Exact large ideals of $B(G)$ are downward directed}
\author[Kaliszewski]{S. Kaliszewski}
\address{School of Mathematical and Statistical Sciences
\\Arizona State University
\\Tempe, Arizona 85287}
\email{kaliszewski@asu.edu}
\author[Landstad]{Magnus~B. Landstad}
\address{Department of Mathematical Sciences\\
Norwegian University of Science and Technology\\
NO-7491 Trondheim, Norway}
\email{magnusla@math.ntnu.no}
\author[Quigg]{John Quigg}
\address{School of Mathematical and Statistical Sciences
\\Arizona State University
\\Tempe, Arizona 85287}
\email{quigg@asu.edu}

\date{December 16, 2015}


\begin{abstract}
We prove that if $E$ and $F$ are large ideals of $B(G)$ for which the associated coaction functors are exact, then the same is true for $E\cap F$.
We also give an example of a coaction functor whose restriction to the maximal coactions does not come from any large ideal.
\end{abstract}
\maketitle

\section{Introduction}\label{intro}

In \cite{BauGueWil} Baum, Guentner, and Willett,
striving to make advances in the Baum-Connes conjecture,
studied \emph{crossed-product functors} $\sigma$ that take an action $(A,\alpha)$ of a locally compact group $G$ to a $C^*$-algebra $A\rtimes_{\alpha,\sigma} G$ lying between the full and reduced crossed products.
It is particularly important to know when $\sigma$ is \emph{exact} in the sense that it preserves short exact sequences.
Motivated by this, in \cite{klqfunctor} we introduced \emph{coaction functors}, 
a certain type of functor on the category of coactions of $G$.
Every coaction functor gives rise to a crossed-product functor by composing with the full-crossed-product functor.
Among other things, we showed that if the coaction functor is exact then so is the associated crossed-product functor.
We paid particular attention to the coaction functors $\tau_E$ coming from \emph{large ideals} $E$ of the Fourier-Stieltjes algebra $B(G)$.
An obvious question is, ``For which large ideals $E$ is the coaction functor $\tau_E$ exact?''
In the current paper we will call $E$ \emph{exact} if $\tau_E$ is exact;
for example, $B(G)$ is exact, but the reduced Fourier-Stieltjes algebra $B_r(G)$ is exact if and only if $G$ is an exact group.
In \cite[Remark~6.23]{klqfunctor} we asked whether the intersection of two exact large ideals is exact,
and we mentioned that we had an idea of how to proceed, and promised to address the question in future work.
In the current paper we fulfill that promise in \thmref{intersect}.

In \cite{klqfunctor} we speculated that the proof would require a ``somewhat more elaborate version of Morita compatibility'', and that it would ``perhaps resemble the property that Buss, Echterhoff, and Willett
call \emph{correspondence functoriality} (see \cite[Theorem~4.9]{BusEchWil})''.
It transpires that we ended up doing something slightly different: rather than change our definition of Morita compatibility, we instead combine it with another concept from \cite{BusEchWil}, namely the \emph{ideal property}.

We also answer another question left open in \cite[Question~6.20]{klqfunctor}: there we asked whether every coaction functor, when restricted to the maximal coactions, is naturally isomorphic to one coming from a large ideal. In \exref{no E} we give a counterexample, stealing a trick from \cite{BusEchWil}.

We wish to thank the referee for suggestions that improved this paper.

\section{Preliminaries}\label{prelims}

We briefly recall a few definitions from \cite{klqfunctor}.
In the \emph{classical category} $\cs$ of $C^*$-algebras, the morphisms are homomorphisms between the $C^*$-algebras themselves, not involving multipliers,
and in the \emph{classical category} $\co$ of coactions the morphisms are morphisms in $\cs$ that are equivariant for the coactions.
Since we are interested in the classical category instead of the nondegenerate one (involving nondegenerate homomorphisms into multiplier algebras), we regard maximalization $(A,\delta)\mapsto (A^m,\delta^m)$ and normalization $(A,\delta)\mapsto (A^n,\delta^n)$ as functors on $\co$ (and we use the notation $\phi^m$ and $\phi^n$ for the respective images of a morphism $\phi$).

We assume that we have fixed once and for all a maximalization functor
$(A,\delta)\mapsto (A^m,\delta^m)$
and a normalization functor
$(A,\delta)\mapsto (A^n,\delta^n)$
on the classical category of coactions,
with canonical equivariant surjections $q^m_A:A^m\to A$
and $\Lambda_A:A\to A^n$.
Recall from \cite[Definition~4.1]{klqfunctor} that a \emph{coaction functor} is a functor $\tau$ on the classical category of coactions,
together with a natural transformation $q^\tau$ from maximalization to $\tau$ such that for each coaction $(A,\delta)$ the homomorphism $q^\tau_A:A^m\to A^\tau$ is surjective and has kernel contained in the kernel of the canonical map $\Lambda_{A^m}:A^m\to A^n$ (which is both a normalization of $(A^m,\delta^m)$ and a maximalization of $(A^n,\delta^n)$).
Maximalization, normalization, and the identity functor are all coaction functors.
There are other known coaction functors, determined by \emph{large ideals} of the Fourier-Stieltjes algebra $B(G)$ (see \cite[Section~6]{klqfunctor}).
Recall from \cite[Definition~3.1]{exotic} that we say an ideal $E$ of $B(G)$ is \emph{large} if it is weak* closed, $G$-invariant, and nonzero (in which case it will necessarily contain $B_r(G)$, by \cite[Lemma~3.14]{graded}).
In \exref{no E} we adapt a construction from \cite{BusEchWil} (who studied crossed-product functors defined on a category of actions) to define new coaction functors not of the preceding types.

In \cite[Definition~4.10]{klqfunctor} we defined a coaction functor to be \emph{exact} if it preserves short exact sequences.

Let $(A,\delta)$ and $(B,\epsilon)$ be coactions,
and let $(X,\zeta)$ be an $(A,\delta)-(B,\epsilon)$ imprimitivity-bimodule coaction.
\cite[Lemma~4.15]{klqfunctor} gives an
$(A^m,\delta^m)-(B^m,\epsilon^m)$ imprimitivity-bimodule coaction $(X^m,\zeta^m)$
such that
\[
X^m\dashind\ker q^m_B=\ker q^m_A
\]
(see \cite[Lemma~4.21]{klqfunctor} for the latter).
In \cite[Definition~4.16]{klqfunctor} we defined a coaction functor $\tau$ to be \emph{Morita compatible} if
for every $(X,\zeta)$ as above
we also have
\[
X^m\dashind\ker q^\tau_B=\ker q^\tau_A.
\]
Trivially, maximalization is Morita compatible,
and by \cite[Proposition~6.10]{klqfunctor} every coaction functor coming from a large ideal is Morita compatible.

Recall that in \cite[Definition~7.2]{klqproper}
we called a coaction $(A,\delta)$ of $G$
\emph{w-proper}
(and in \cite[Definition~5.1]{exotic} we used the term ``slice proper'')
if
$(\omega\otimes\id)\circ\delta(A)\subset C^*(G)$ for all $\omega\in A^*$.

If $(A,\delta)$ is a coaction, we call an ideal $I$ of $A$ \emph{strongly invariant} (see, e.g.,\cite[Definition~3.16]{klqfunctor} if
\[
\clspn\{\delta(I)(1\otimes C^*(G))\}=I\otimes C^*(G).
\]
Note that this is precisely what is needed for the restriction of $\delta$ to $I$ to be a coaction.

\section{Main result}\label{main}

We recall a few definitions
from \cite[Section~6]{klqfunctor}:
given any coaction $(A,\delta)$ and any large ideal $E$ of $B(G)$,
we define an ideal
\[
A_E=\{a\in A:E\cdot a=\{0\}\},
\]
and we write $A^E=A/A_E$ for the quotient $C^*$-algebra.
The quotient map $Q^E_A:A\to A_E$ is equivariant for $\delta$ and a coaction $\delta^E$ on $A^E$,
and $(A,\delta)\mapsto (A^E,\delta^E)$ is a coaction functor that we denote by $\tau_E$.

\begin{defn}
We call a large ideal $E$ of $B(G)$ \emph{exact} if the associated coaction functor $\tau_E$ is exact.
\end{defn}

We will prove that the set of exact large ideals of $B(G)$ is downward directed by showing that it is in fact closed under finite intersections. By induction, \thmref{intersect} below does the job.
It remains an open question whether the intersection of all exact large ideals of $B(G)$ is exact.

\begin{thm}\label{intersect}
The intersection of two exact large ideals of $B(G)$ is exact.
\end{thm}

The key idea of our proof is
the following:
for two large ideals $E$ and $F$ of $B(G)$,
we compare the intersection $E\cap F$ to
the product. The following definition makes this precise:

\begin{defn}
For two large ideals $E,F\subset B(G)$ we write
$\<EF\>$ for the weak*-closed linear span of the set $EF$ of products.
\end{defn}

\begin{rem}
It is somewhat frustrating that we do not know of any examples of exact large ideals other than $B(G)$
(and, when $G$ is exact, $B_r(G)$).
Perhaps other examples could be found using techniques similar to those of
\cite[Section 5]{BauGueWil}.
\end{rem}

Note that $\<EF\>$ is a large ideal of $B(G)$ contained in the intersection $E\cap F$.
In \cite[Corollary~6.9]{klqfunctor} we showed that if $E$ or $F$ is exact then $\<EF\>=E\cap F$.
On the other hand, in \cite[proof of Proposition~8.4]{exotic} we observed that it follows from work of \cite{okayasu} that if $G$ is a noncommutative free group and $E_p$ is the weak*-closure in $B(G)$ of $\spn\{P(G)\cap L^p(G)\}$, where $P(G)$ denotes the set of positive type functions on $G$, then for for every $p>2$ we have
\[
\<E_p^2\>\subset E_{p/2}\subsetneq E_p.
\]
Note that in \cite[Section~8]{exotic}, $E_p$ was defined using $B(G)\cap L^p(G)$;
it now seems clear that this should be changed to $\spn\{P(G)\cap L^p(G)\}$ ---
see \cite[Proposition~2.13]{bew2}. We are grateful to Buss, Echterhoff, and Willett for pointing this out to us.

Another key idea in our strategy is to first do it for w-proper coactions.
Although w-properness is a quite strong hypothesis,
in some sense it is not:

\begin{lem}\label{morita}
Every coaction is Morita equivalent to a w-proper one.
\end{lem}

\begin{proof}
Let $(A,\delta)$ be a coaction,
with maximalization $(A^m,\delta^m)$.
Since $(A^m,\delta^m)$ is maximal,
the double crossed product gives
a 
coaction $(B,\epsilon)$,
an $A^m-B$ imprimitivity bimodule $X$,
and a $\delta^m-\epsilon$ compatible coaction $\zeta$ on $X$.
By \cite[Corollary~7.8]{klqproper}, $(B,\epsilon)$ is w-proper since it is a dual coaction.
Let $I$ be the kernel of the maximalization map $q^m_A:A^m\to A$,
and let $J$ be the ideal of $B$ induced via the imprimitivity bimodule $X$.
Since the imprimitivity bimodule $X$ is equivariant,
there is a coaction $\wilde\epsilon$ on the quotient $B/J$
such that the given coaction $(A,\delta)$ is Morita equivalent to $(B/J,\wilde\epsilon)$.
By \cite[Proposition~5.3]{exotic},
the coaction $\wilde\epsilon$ is w-proper.
\end{proof}

\begin{lem}\label{EF}
If $E$ and $F$ are large ideals of $B(G)$, 
then for every w-proper coaction $(A,\delta)$
there is a unique isomorphism $\theta_A$ making the diagram
\[
\xymatrix@C+20pt{
A \ar[r]^-{Q^E_A} \ar[d]_{Q^{\<EF\>}_A}
&A^E \ar[d]^{Q^F_{A^E}}
\\
A^{\<EF\>} \ar@{-->}[r]_-{\theta_A}^-\simeq
&(A^E)^F.
}
\]
commute.
\end{lem}

\begin{proof}
We will show that
$\ker Q^F_{A^E}=A_{\<EF\>}/A_E$.
Since $\ker Q^{\<EF\>}_A=A_{\<EF\>}$, this will imply that
\[
\ker Q^{\<EF\>}_A=\ker Q^F_{A^E}\circ Q^E_A,
\]
and the result will follow.
For all $a+A_E\in A^E=A/A_E$
we have
$a+A_E\in \ker Q^F_{A^E}=(A^E)_F$ if and only if
$F\cdot (a+A_E)=\{A_E\}$,
equivalently $F\cdot a\subset A_E$,
equivalently $EF\cdot a=\{0\}$.
By definition,
$\<EF\>$ is the weak*-closed span of $EF$.
Since $\delta$ is w-proper,
the map
\[
f\mapsto f\cdot a:B(G)\to A
\]
is weak*-to-weakly continuous,
so
$EF\cdot a=\{0\}$ if and only if $\<EF\>\cdot a=\{0\}$, i.e., $a\in A_{\<EF\>}$.
Thus
\[
\ker Q^F_{A^E}=(A^E)_F=A_{\<EF\>}/A_E.
\qedhere
\]
\end{proof}

The following result almost shows that the $\theta_A$ of \lemref{EF} gives a natural isomorphism between the coaction functors $\tau_{\<EF\>}$ and $\tau_F\circ \tau_E$:

\begin{lem}\label{natural}
Let $E$ and $F$ be large ideals of $B(G)$.
Let $(A,\delta)$ and $(B,\epsilon)$ be w-proper coactions,
and let $\psi:A\to B$ be a $\delta-\epsilon$ equivariant homomorphism.
Then the diagram
\begin{equation}\label{natural diagram}
\xymatrix@C+30pt{
A^{\<EF\>} \ar[r]^-{\psi^{\<EF\>}} \ar[d]_{\theta_A}
&B^{\<EF\>} \ar[d]^{\theta_B}
\\
(A^E)^F \ar[r]_-{(\psi^E)^F}
&(B^E)^F
}
\end{equation}
commutes
equivariantly for the appropriate coactions.
\end{lem}

\begin{proof}
Equation~
\eqref{natural diagram} is the outer square of the following diagram:
\[
\xymatrix@C+30pt{
A^{\<EF\>} \ar[rrr]^-{\psi^{\<EF\>}} \ar[ddd]_{\theta_A}
&&&B^{\<EF\>} \ar[ddd]^{\theta_B}
\\
&A \ar[r]^-\psi \ar[d]_{Q^E_A} \ar[ul]^{Q^{\<EF\>}_A}
&B \ar[d]^{Q^E_B} \ar[ur]_{Q^{\<EF\>}_B}
\\
&A^E \ar[r]_-{\psi^E} \ar[dl]_{Q^F_{A^E}}
&B^E \ar[dr]^{Q^F_{B^E}}
\\
(A^E)^F \ar[rrr]_-{(\psi^E)^F}
&&&(B^E)^F.
}
\]
The left and right quadrilaterals commute by \lemref{EF}.
The top, middle, and bottom quadrilaterals commute by functoriality.
Since $Q^{\<EF\>}_A$ is surjective, it follows that the outer square commutes.
Since all maps except possibly for $\theta_A$ and $\theta_B$ are equivariant for appropriate
coactions, the isomorphisms $\theta_A$ and $\theta_B$ are also equivariant.
\end{proof}

\begin{defn}
Let $\tau$ be a coaction functor.
We say that a coaction $(A,\delta)$ is \emph{$\tau$-exact} if for every strongly $\delta$-invariant ideal $I$ of $A$ the sequence
\[
\xymatrix{
0\ar[r]
&I^\tau \ar[r]
&A^\tau \ar[r]
&(A/I)^\tau \ar[r]
&0
}
\]
is exact.
\end{defn}
Thus, a coaction functor $\tau$ is exact if and only if every coaction is $\tau$-exact.

\begin{lem}\label{w-proper}
If $E$ and $F$ are exact large ideals of $B(G)$, then every w-proper coaction is $\tau_{\<EF\>}$-exact.
\end{lem}

\begin{proof}
Let $(A,\delta)$ be a w-proper coaction, and let $I$ be a strongly $\delta$-invariant ideal of $A$.
Then we have an equivariant short exact sequence
\[
\xymatrix{
0\ar[r]
&I\ar[r]^{\phi}
&A\ar[r]^\psi
&B\ar[r]
&0,
}
\]
where $\phi$ is the inclusion, $B=A/I$, and $\psi$ is the quotient map.
We must show that 
the sequence
\begin{equation}\label{D seq}
\xymatrix{
0\ar[r]
&I^{\<EF\>}\ar[r]^{\phi^{\<EF\>}}
&A^{\<EF\>}\ar[r]^{\psi^{\<EF\>}}
&B^{\<EF\>}\ar[r]
&0
}
\end{equation}
is exact.

Since $E$ is exact,
the sequence
\[
\xymatrix{
0\ar[r]
&I^E\ar[r]^{\phi^E}
&A^E\ar[r]^{\psi^E}
&B^E\ar[r]
&0
}
\]
is exact.
Then since $F$ is exact,
the sequence
\[
\xymatrix@C+20pt{
0\ar[r]
&(I^E)^F\ar[r]^{(\phi^E)^F}
&(A^E)^F\ar[r]^{(\psi^E)^F}
&(B^E)^F\ar[r]
&0
}
\]
is exact.

By \lemref{natural}, we have an isomorphism
\[
\xymatrix@C+20pt@R+10pt{
0\ar[r]
&I^{\<EF\>}\ar[r]^{\phi^{\<EF\>}} \ar[d]_{\theta_I}^\simeq
&A^{\<EF\>}\ar[r]^{\psi^{\<EF\>}} \ar[d]_{\theta_A}^\simeq
&B^{\<EF\>}\ar[r] \ar[d]^{\theta_B}_\simeq
&0
\\
0\ar[r]
&(I^E)^F\ar[r]_{(\phi^E)^F}
&(A^E)^F\ar[r]_{(\psi^E)^F}
&(B^E)^F\ar[r]
&0
}
\]
of sequences,
so the top sequence is exact since the bottom one is.
\end{proof}

The following is adapted from \cite[Definition~3.1]{BusEchWil}.

\begin{defn}
We say a coaction functor $\tau$ has the \emph{ideal property}
if for every coaction $(A,\delta)$ and
every strongly $\delta$-invariant ideal $I$ of $A$,
letting
$\iota:I\hookrightarrow A$ denote the inclusion map, the induced map
\[
\iota^\tau:I^\tau\to A^\tau
\]
is injective.
\end{defn}

Note that in the above definition, if $\tau$ has the ideal property then the image of $I^\tau$ in $A^\tau$ will be a strongly $\delta^\tau$-invariant ideal,
and we will identify $I^\tau$ with this image, regarding it as an ideal of $A^\tau$.

\cite[Remark~3.4]{BusEchWil} says that the ideal property holds for every crossed-product functor coming from a large ideal. This also follows from the following lemma.

\begin{lem}\label{E ideal property}
For every large ideal $E$ of $B(G)$ the coaction functor $\tau_E$ has the ideal property.
\end{lem}

\begin{proof}
This follows from \cite[Proof of Proposition~6.7]{klqfunctor}, where it is shown that \cite[Equation~(6.4)]{klqfunctor} holds automatically.
\end{proof}

\begin{rem}
Every exact coaction functor has the ideal property, but normalization is a coaction functor that is not exact but nevertheless has the ideal property.
We do not know an example of a decreasing coaction functor that is Morita compatible and does not have the ideal property.
\end{rem}

\begin{prop}\label{morita preserves}
Let $\tau$ be a Morita compatible coaction functor with the ideal property,
and let $(A,\delta)$ and $(B,\epsilon)$ be Morita equivalent coactions.
Then $(A,\delta)$ is $\tau$-exact if and only if $(B,\epsilon)$ is.
\end{prop}

\begin{proof}
Let $X$ be an equivariant $A-B$ imprimitivity bimodule.
Without loss of generality assume that $(B,\epsilon)$ is $\tau$-exact,
let $I$ be an invariant ideal of $A$,
let $J$ be the ideal of $B$ corresponding to $I$ via $X$,
and let
\begin{align*}
\psi_A&:A\to A/I
\\
\psi_B&:B\to B/J
\end{align*}
be the quotient maps.
We know that $\psi_B^\tau:B^\tau\to (B/J)^\tau$ is surjective and has kernel $J^\tau$, because we are assuming that $(B,\epsilon)$ is $\tau$-exact.

Since maximalization is an exact coaction functor \cite[Theorem~4.11]{klqfunctor}, the homomorphism
\[
\psi_A^m:A^m\to (A/I)^m
\]
is surjective.
Since $q^\tau$ is a natural transformation from maximalization to $\tau$,
the diagram
\[
\xymatrix@C+30pt@R+20pt{
A^m \ar[r]^-{\psi_A^m} \ar[d]_{q^\tau_A}
&(A/I)^m \ar[d]^{q^\tau_{A/I}}
\\
A^\tau \ar[r]_{\psi^\tau_A}
&(A/I)^\tau
}
\]
commutes.
Thus $\psi^\tau_A$ is surjective.
Since $\tau$ has the ideal property,
$I^\tau$ is an ideal of $A^\tau$.
For $\tau$-exactness of $(A,\delta)$, it remains to show that $I^\tau=\ker\psi^\tau_A$.

Since $\tau$ is Morita compatible, by \cite[Lemma~4.19]{klqfunctor} we have an equivariant $A^\tau-B^\tau$ imprimitivity bimodule $X^\tau$ and a surjective $q^\tau_A-q^\tau_B$ compatible imprimitivity-bimodule homomorphism $q^\tau_X:X^m\to X^\tau$, where $X^m$ is the equivariant $A^m-B^m$ imprimitivity bimodule of \cite[Lemma~4.14]{klqfunctor}.
(Note that \cite[Lemma~4.19]{klqfunctor} did not explicitly mention surjectivity of $q^\tau_X$, but this surjectivity follows from that of $q^\tau_A$ and $q^\tau_B$.)
We visualize this using the diagram
\[
\xymatrix@R+20pt{
A^m \ar@{-}[r] \ar[d]_{q^\tau_A}
&X^m \ar@{-}[r] \ar[d]_{q^\tau_X}
&B^m \ar[d]_{q^\tau_B}
\\
A^\tau \ar@{-}[r]
&X^\tau \ar@{-}[r]
&B^\tau
}
\]
Similarly for $q^\tau_{X/Y}:(X/Y)^m\to (X/Y)^\tau$.

Consider the diagram
\begin{equation}\label{big diagram}
\xymatrix@R+30pt{
A^m \ar@{-}[r] \ar[dd]_{q^\tau_A} \ar[drrr]|{\psi_A^m}
&X^m \ar@{-}[r] \ar[dd]_{q^\tau_X} \ar[drrr]|{\psi_X^m}
&B^m \ar[dd]_{q^\tau_B} \ar[drrr]|{\psi_B^m}
\\
&&&(A/I)^m \ar@{-}[r] \ar[dd]_(.3){q^\tau_{A/I}}
&(X/Y)^m \ar@{-}[r] \ar[dd]_(.3){q^\tau_{X/Y}}
&(B/J)^m \ar[dd]_(.3){q^\tau_{B/J}}
\\
A^\tau \ar@{-}[r] \ar[drrr]|{\psi_A^\tau}
&X^\tau \ar@{-}[r] \ar@{-->}[drrr]|{\psi_X^\tau}
&B^\tau \ar[drrr]|{\psi_B^\tau}
\\
&&&(A/I)^\tau \ar@{-}[r]
&(X/Y)^\tau \ar@{-}[r]
&(B/J)^\tau.
}
\end{equation}
Claim:
there is an imprimitivity-bimodule homomorphism $\psi^\tau_X$ as \eqref{big diagram} indicates, with coefficient homomorphisms $\psi^\tau_A$ and $\psi^\tau_B$.
To get a \emph{linear map} $\psi^\tau_X$ such that the diagram
\[
\xymatrix{
X^m \ar[r]^-{\psi_X^m} \ar[d]_{q^\tau_X}
&(X/Y)^m \ar[d]^{q^\tau_{X/Y}}
\\
X^\tau \ar@{-->}[r]_-{\psi^\tau_X}
&(X/Y)^\tau
}
\]
commutes,
it suffices to show that $\ker q^\tau_X\subset \ker q^\tau_{X/Y}\circ \psi^m_\tau$.
Suppose $x\in \ker q^\tau_X$.
Since $\ker q^\tau_X=X^m\cdot \ker q^\tau_B$,
by the Cohen-Hewitt factorization theorem
we can factor $x=x'\cdot b$,
where $b\in \ker q^\tau_B$.
Then
\begin{align*}
&q^\tau_{X/Y}\circ \psi_X^m(x)
\\&\quad=q^\tau_{X/Y}\circ \psi_X^m(x'\cdot b)
\\&\quad=q^\tau_{A/I}(x')\cdot q^\tau_{B/J}\circ \psi^m_B(b)
\\&\hspace{.5in}\text{(since $q^\tau_{X/Y}$ and $\psi^m_X$ are imprimitivity-bimodule homomorphisms)}
\\&\quad=q^\tau_{A/I}(x')\cdot \psi^\tau_B\circ q^\tau_B(b)
\righttext{(by naturality of $q^\tau$)}
\\&\quad=0,
\end{align*}
as desired.
The computations required to verify that the linear map $\psi^\tau_X$ is an imprimitivity-bimodule homomorphism are routine:
for the right-module structures,
let $x\in X^\tau$ and $b\in B^\tau$.
By surjectivity we can write 
$x=q^\tau_X(x')$ and $b=q^\tau_B(b')$
with $x'\in X^m$ and $b'\in B^m$, and then
\begin{align*}
\psi^\tau_X(x\cdot b)
&=\psi^\tau_X\bigl(q^\tau_X(x')\cdot q^\tau_B(b')\bigr)
\\&=\psi^\tau_X\circ q^\tau_X(x'\cdot b')
\\&=q^\tau_{X/Y}\circ \psi^m_X(x'\cdot b')
\\&\overset{*}=q^\tau_{X/Y}\circ \psi^m_X(x')\cdot q^\tau_{B/J}\circ \psi^m_B(b')
\\&=\psi^\tau_X\circ q^\tau_X(x')\cdot \psi^\tau_B\circ q^\tau_B(b')
\\&=\psi^\tau_X(x)\cdot \psi^\tau_B(b),
\end{align*}
where the equality at $*$ follows since $q^\tau_{X/Y}$ and $\psi^m_X$ are imprimitivity-bimodule homomorphisms.
Similarly for the left-module structures.
For the right-hand inner products,
let $x,y\in X^\tau$.
Factor $x=q^\tau_X(x')$ and $y=q^\tau_X(y')$ with $x',y'\in X^m$.
Then
\begin{align*}
\psi^\tau_A\bigl({}_{A^\tau}\<x,y\>\bigr)
&=\psi^\tau_A\bigl({}_{A^\tau}\<q^\tau_X(x'),q^\tau_X(y')\>\bigr)
\\&=\psi^\tau_A\circ q^\tau_A\bigl({}_{A^m}\<x',y'\>\bigr)
\\&=q^\tau_{A/I}\circ \psi^m_A\bigl({}_{A^m}\<x',y'\>\bigr)
\\&={}_{(A/I)^\tau}\bigl\<q^\tau_{X/Y}\circ \psi^m_X(x'),q^\tau_{X/Y}\circ \psi^m_X(y')\bigr\>
\\&={}_{(A/I)^\tau}\bigl\<\psi^\tau_X\circ q^\tau_X(x'),\psi^\tau_X\circ q^\tau_X(y')\bigr\>
\\&={}_{(A/I)^\tau}\bigl\<\psi^\tau_X(x),\psi^\tau_X(y)\bigr\>,
\end{align*}
and similarly for the right-hand inner products.
Thus $\psi^\tau_X$ is an imprimitivity-bimodule homomorphism
with coefficient homomorphisms $\psi^\tau_A$ and $\psi^\tau_B$, proving the claim.

It now follows from \cite[Lemma~1.20]{enchilada} that
\[
\ker \psi^\tau_A=X^\tau-\ind \ker \psi^\tau_B=X^\tau-\ind J^\tau.
\]
Since we also have the imprimitivity-bimodule homomorphism $\psi^m_X$
with coefficient homomorphisms $\psi^m_A$ and $\psi^m_B$,
and since maximalization is an exact coaction functor,
\[
I^m=\ker \psi^m_A=X^m-\ind \ker \psi^m_B=X^m-\ind J^m.
\]
Now, $\delta$ restricts to a coaction $(I,q_I)$,
and by surjectivity of $q^\tau$ we have
\[
I^\tau=q^\tau_B(I^m),
\]
and similarly $J^\tau=q^\tau_B(J^m)$.
Combining, we get
\begin{align*}
I^\tau
&=q^\tau_A(I^m)
\\&=q^\tau_A(X^m-\ind J^m)
\\&=X^\tau-\ind q^\tau_B(J^m)
\\&\hspace{1in}\text{(since $q^\tau_X$ is an imprimitivity-bimodule homomorphism)}
\\&=X^\tau-\ind J^\tau
\\&=X^\tau-\ind \ker \psi^\tau_B
\\&=\ker \psi^\tau_A,
\end{align*}
finishing the proof.
\end{proof}

\begin{proof}[Proof of \thmref{intersect}]
Let $E$ and $F$ be exact large ideals.
By \cite[Corollary~6.9]{klqfunctor} we have $E\cap F=\<EF\>$,
and by \lemref{E ideal property} the coaction functor $\tau_{\<EF\>}$ has the ideal property.
Further, by \cite[Proposition~6.10]{klqfunctor} $\tau_{\<EF\>}$ is Morita compatible.
The conclusion now follows from \lemref{w-proper}, \propref{morita preserves}, and \lemref{morita}.
\end{proof}

\begin{rem}
The technique of proof of \cite[Theorem~4.22]{klqfunctor} shows that
the greatest lower bound of any collection of exact coaction functors is exact. Thus, it might seem that \thmref{intersect} above implies that the intersection $E$ of all exact
large ideals of $B(G)$ is exact. However, it is not clear to us
how to show that $\tau_E$ coincides with the greatest lower bound of $\{\tau_F:\text{$F$ is a exact large ideal}\}$; it is certainly no larger than
this greatest lower bound, but that is all we can prove at this point.
To see what the problem is, let $\{E_i\}$ be the set of exact large ideals of $B(G)$,
so that $E=\bigcap_iE_i$.
The issue is whether, for a given coaction $(A,\delta)$,
the union $\bigcup_iA_{E_i}$ of the upward-direct family of ideals is dense in the ideal $A_E$.
This is true for $(C^*(G),\delta_G)$ since then $A_E=\pann E$
and $\bigcup_i\pann E_i$ is dense in $\pann E$
because
$(\bigcup_i\pann E_i)\ann=\bigcap_i(\pann E_i)\ann=E$.
In the general case,
we have $(A_E)\ann=\clspn\{EA^*\}$ (the weak*-closure of the linear span of products, where $E$ acts on the dual space $A^*$ in the natural way).
Obviously
$EA^*
\subset \bigcap_i\clspn\{E_iA^*\}$,
but we cannot see a reason to expect $\spn \{EA^*\}$ to be weak*-dense in this intersection.
\end{rem}

\begin{rem}
\cite[Subsection~9.2 Question~(1)]{BusEchWil} asks whether, for every exact group $G$
and all $p\in [2,\infty)$,
the crossed-product functor $\rtimes_{E_p}$ is exact,
where $E_p$ is the weak*-closure of $B(G)\cap L^p(G)$ (which should be changed to $\spn\{P(G)\cap L^p(G)\}$, as in the discussion preceding \lemref{morita} of the current paper and in \cite[Proposition~2.13]{bew2}).
We know that if $G$ is a free group $\F_n$ with $n>1$,
then for $2\le p<\infty$ the coaction functor $\tau_{E_p}$ is not exact.
Of course, $\F_n$ is exact.
We think we might be able to deduce that $\rtimes_{E_p}$ is not exact.
Note that this is nontrivial: if we compose a coaction functor $\tau$ with the full-crossed-product functor $\cp$ that takes an action $(B,\alpha)$ to the dual coaction $(B\rtimes_\alpha G,\what\alpha)$,
we get a crossed-product functor
$\mu_\tau:=\tau\circ \cp$
that takes $(B,\alpha)$ to the coaction
\[
\bigl((B\rtimes_\alpha G)^\tau,(\what\alpha)^\tau\bigr).
\]
By \cite[Proposition~4.24]{klqfunctor},
if $\tau$ 
is exact or Morita compatible
then so is $\mu_\tau$.
But to give a negative answer to the \cite{BusEchWil} question we would be trying to draw a conclusion that goes in the ``wrong direction''.
\end{rem}

\begin{ex}\label{no E}
\cite[Question~6.20]{klqfunctor} asks whether for every coaction functor $\tau$
there necessarily exists a large ideal $E$ of $B(G)$ such that
the restrictions of $\tau$ and $\tau_E$ to the subcategory of maximal coactions (but still taking values in the ambient category of coactions)
are naturally isomorphic.
Borrowing a trick from Buss, Echterhoff, and Willett,
we give a negative answer.
We adapt a construction from \cite[Section~2.5 and Example~3.5]{BusEchWil}.
Let $\RR$ be a collection of coactions.
For each coaction $(A,\delta)$ let
$\RR_{A,\delta}$ be the collection of all triples $(B,\epsilon,\phi)$,
where either
$(B,\epsilon)\in\RR$ and $\phi:(A,\delta)\to (B,\epsilon)$ is a morphism in $\co$
or
$(B,\epsilon)=(A^n,\delta^n)$ and $\epsilon:(A,\delta)\to (A^n,\delta^n)$ is the normalization 
surjection.
Then let
\[
\left(\bigoplus_{(B,\epsilon,\phi)\in\RR_{A,\delta}}(B,\epsilon),
\bigoplus_{(B,\epsilon,\phi)\in\RR_{A,\delta}}\epsilon\right)
\]
be the direct-sum coaction.
We can form the direct sum
\[
Q^\RR_A:=
\bigoplus_{(B,\epsilon,\phi)\in\RR_{A,\delta}}\phi:
A\to M\left(\bigoplus_{(B,\epsilon,\phi)\in\RR_{A,\delta}}B\right),
\]
which is a nondegenerate
\[
\delta-\bigoplus_{(B,\epsilon,\phi)\in\RR_{A,\delta}}\epsilon
\]
equivariant homomorphism.
Let $A^\RR$ be the image of $A$ under this direct-sum homomorphism $Q^\RR_A$.
Then by the elementary \lemref{image} below there is a unique coaction $\delta^\RR$ of $G$ on $A^\RR$
such that $Q^\RR_A$ is $\delta-\delta^\RR$ equivariant.

Claim: for every morphism $\phi:(A,\delta)\to (B,\epsilon)$ in $\co$,
there is a unique morphism $\phi^\RR$ in $\co$ making the diagram
\[
\xymatrix@C+20pt{
(A,\delta) \ar[r]^-\phi \ar[d]_{Q^\RR_A}
&(B,\epsilon) \ar[d]^{Q^\RR_B}
\\
(A^\RR,\delta^\RR) \ar@{-->}[r]_-{\phi^\RR}
&(B^\RR,\epsilon^\RR)
}
\]
commute.
We have
\[
\ker Q^\RR_A=\bigcap_{(C,\eta,\psi)\in\RR_{A,\delta}}\ker\psi,
\]
while
\[
\ker Q^\RR_B\circ\phi
=\bigcap_{(C,\eta,\psi)\in\RR_{B,\epsilon}}\ker\psi\circ\phi.
\]
Since
\[
\{\psi\circ\phi:(C,\eta,\psi)\in\RR_{B,\epsilon}\}\subset \RR_{A,\delta},
\]
we have
\[
\ker Q^\RR_A\subset \ker Q^\RR_B\circ\phi,
\]
and the claim follows.

Uniqueness of the maps $\phi^\RR$ and surjectivity of the maps $Q^\RR_A$
implies that there is a unique decreasing coaction functor $\tau_\RR$ such that
\[
(A^{\tau_\RR},\delta^{\tau_\RR})=(A^\RR,\delta^\RR)
\]
and $\phi^{\tau_\RR}=\phi^\RR$
(see \cite[Definition~5.1 and Lemma~5.2]{klqfunctor}).

We will show that,
whenever $G$ is nonamenable,
there is a suitable choice of $\RR$ for which the coaction functor $\tau_\RR$ is not Morita compatible, and therefore its restriction to the maximal coactions is not naturally isomorphic to $\tau_E$ for any large ideal $E$ of $B(G)$.
Let
\[
(A,\delta)=\bigl(C[0,1)\otimes C^*(G),\id\otimes\delta_G\bigr).
\]
We let
\[
\RR=\{(A,\delta)\}.
\]
The coactions $(A,\delta)$ and $(\KK\otimes A,\id\otimes\delta)$ are Morita equivalent.
We claim that
$Q^\RR_A$ is faithful but
$Q^\RR_{\KK\otimes A}$ is not.
Since the coaction functor $\tau_\RR$ is decreasing, it will follow that $\tau_\RR$ is not Morita compatible.

The triple $(A,\delta,\id)$ is in the collection $\RR_{A,\delta}$, which implies that $Q^\RR_A$ is faithful.
On the other hand,
we claim that the only morphism in the collection $\RR_{\KK\otimes A,\id\otimes\delta}$ is the normalization
\[
\id_{\KK\otimes C[0,1)}\otimes\lambda:\KK\otimes C[0,1)\otimes C^*(G)
\to
\KK\otimes C[0,1)\otimes C^*_r(G).
\]
Since $G$ is nonamenable, this normalization is not faithful.
To verify the claim,
it will suffice to show that there are no nonzero
homomorphisms
from
$\KK\otimes A$
to
$A$.
Any such homomorphism would be of the form
$\psi_1\times\psi_2$, where $\psi_1$ and $\psi_2$ are commuting homomorphisms from $\KK$ and $A$, respectively, to $A$, i.e.,
\[
(\psi_1\times\psi_2)(k\otimes a)=\psi_1(k)\psi_2(a).
\]
Since $A$ has no nonzero projections, the homomorphism $\psi_1$ must be 0, and so $\psi_1\times\psi_2=0$.
\end{ex}

In \exref{no E}, we used the following lemma, which is presumably folklore.
Since we could not find it in the literature we include the proof.

\begin{lem}\label{image}
Let $(A,\delta)$ and $(B,\epsilon)$ be coactions,
and let $\phi:A\to M(B)$ be a $\delta-\epsilon$ equivariant homomorphism.
Let $C=\phi(A)\subset M(B)$.
Then there is a unique coaction $\eta$ of $G$ on $C$ such that $\phi:A\to C$ is $\delta-\eta$ equivariant.
\end{lem}

\begin{proof}
By \cite[Corollary~1.7]{fullred}, it suffices to show that $C$ is a nondegenerate $A(G)$-submodule of $M(B)$.

Now, \cite[Proposition~A.1]{klqfunctor} says that a homomorphism from $A$ to $B$ is $\delta-\epsilon$ equivariant if and only if it is a $B(G)$-module map.
We need a slight extension of this, namely the case of homomorphisms $\phi:A\to M(B)$.
The argument of \cite[Proposition~A.1]{klqfunctor} carries over, with the minor adjustment that in the second line of the multiline displayed computation 
the map $\phi\otimes\id$ must be replaced by the canonical extension
\[
\bar{\phi\otimes\id}:\wilde M(A\otimes C^*(G))\to M(B\otimes C^*(G)),
\]
which exists by \cite[Proposition~A.6]{enchilada}.
Thus, since we are assuming $\delta-\epsilon$ equivariance,
we can conclude that
\[
A(G)\cdot \phi(A)=\phi\bigl(A(G)\cdot A\bigr).
\]
Since $A$ is a nondegenerate $A(G)$-module, we are done.
\end{proof}


\providecommand{\bysame}{\leavevmode\hbox to3em{\hrulefill}\thinspace}
\providecommand{\MR}{\relax\ifhmode\unskip\space\fi MR }
\providecommand{\MRhref}[2]{%
  \href{http://www.ams.org/mathscinet-getitem?mr=#1}{#2}
}
\providecommand{\href}[2]{#2}

\end{document}